\documentclass[14pt]{amsart}
\pdfoutput=1
\usepackage{cases}
\usepackage{amsmath}
\usepackage{amsfonts}
\usepackage{bm}
\usepackage{amsfonts,amsmath,amssymb,amscd,bbm,amsthm,mathrsfs,dsfont}
\usepackage{mathrsfs}
\usepackage{pb-diagram}
\usepackage{color}
\usepackage{amssymb}
\usepackage{mathtools}
\usepackage{amsthm}
\usepackage[all]{xy}
\usepackage{epsfig}
\usepackage{graphics}
\usepackage{array}
\usepackage{graphicx} 
\usepackage{epstopdf}
\usepackage{float}
\usepackage{epsfig}
\usepackage{tikz}
\usepackage{tikz-cd}
\usetikzlibrary{arrows}
\usetikzlibrary{graphs}
\usepackage{tikz}
\usepackage{hyperref}
\hypersetup{
    colorlinks=true,
    linkcolor=blue,
    filecolor=magenta,
    urlcolor=cyan,
    citecolor=red
}
\usepackage{cleveref}

\newtheorem{Theorem}{Theorem}[section]
\newtheorem{Lemma}[Theorem]{Lemma}
\newtheorem{Definition}[Theorem]{Definition}
\newtheorem{Corollary}[Theorem]{Corollary}
\newtheorem{Proposition}[Theorem]{Proposition}

\newtheorem{Conjecture}[Theorem]{Conjecture}

\title {On the properties of acyclic sign-skew-symmetric cluster algebras}
\author{Siyang Liu}

\address{Siyang Liu
	\newline School of Mathematics,
	Hangzhou Normal  University, Hangzhou 311121, P.R.China}
\email{siyangliu@hznu.edu.cn}

\newcommand{\lra}{\longrightarrow}

\newcommand{\ra}{\rightarrow}
\newcommand{\sdp}{\times\kern-.2em\vrule height1.1ex depth-.05ex}
\newcommand{\epi}{\lra \kern-.8em\ra}
\newcommand{\C}{{\mathbb C}}

\makeatletter
\@addtoreset{equation}{section}
\makeatother

 \normalbaselines
\begin{document}

\maketitle
\bigskip

\tableofcontents

\begin{abstract}
 We  study the tropical dualities and  properties of exchange graphs for the totally  sign-skew-symmetric cluster algebra under a condition.  We prove that  the condition always holds for acyclic cluster algebras, then all results hold for the acyclic case.
\end{abstract}

	\section{Introduction}
Cluster algebras, introduced by	Fomin and Zelevinsky   in  \cite{fomin2002cluster}  for providing a combinatorial framework to study total positivity and canonical bases associated by Lusztig to algebraic groups,  are certain commutative algebras generated by certain combinatorially deﬁned generators (called cluster variables) which are grouped into overlapping clusters. Since cluster algebras were invented, many links between cluster algebras and other subjects have been discovered.
For example, cluster algebras have close relations with Riemann surfaces and triangulations, Zamolodchikov's periodicity properties concerning $Y$-systems,  representation theory,  Donaldson-Thomas invariant theory and scattering diagrams. 

The initial data to construct a cluster algebra is a totally sign-skew-symmetric matrix in the first paper in the series of foundational papers of cluster theory. Since then, most authors 
prefer to study skew-symmetrizable cluster algebras and many important properties and conjectures were proved  in the skew-symmetrizable case. Even though there are still many unkonwns for totally sign-skew-symmetric cluster algebras, we believe  that their properties behave like the skew-symmetrizable case. 

The most important property of cluster algebras is the Laurent phenomenon.  Fomin and Zelevinsky conjectured the cluster variables written as Laurent polynomials of a cluster always have the positive coefficients, which is called the {\em positivity conjecture}. This conjecture was proved by Lee and Schiffler for the symmetric case \cite{LS}, by Gross, Hacking, Keel and Kontsevich  for the skew-symmetrizable case \cite{gross2018canonical}, Huang and Li for the acyclic sign-skew-symmetric cluster algebras. Recently,   Li and Pan proved it  for all totally sign-skew-symmetric case \cite{LP}.

In the fourth paper of the foundational papers \cite{fomin2007cluster},  Fomin and Zelevinsky introduced a lot of combinatorial  objects including $c$-vectors and $g$-vectors (also $C$-matrices and $G$-matrices) to study cluster algebras. The  well-known {\em sign coherence conjecture} was also proposed in \cite{fomin2007cluster}. 
The dualities of $G$-matrix and $C$-matrix were systematically studied by Nakanishi and Zelevinsky in \cite{NZ}. In their paper, they shall assume the sign coherence conjecture holds for $c$-vectors for all skew-symmetrizable cluster algebras. This was confirmed by  Derksen, Weyman, and Zelevinsky in the skew-symmetric case \cite{DWZ} and by  Gross,  Hacking,  Keel, and Kontsevich in the skew-symmetrizable case \cite{gross2018canonical}. Thus the dualities obtained by Nakanishi and Zelevinsky in \cite{NZ} are true for skew-symmetrizable cluster algebras. Let $\mathcal{A}$ be a skew-symmetrizable cluster algebra of rank $n$.
Roughly speaking, they proved the following results for skew-symmetrizable cluster algebras:
\begin{align}
	\label{005} &\text{Every $G$-matrix is an invertible matrix of some $C$-matrix;}\\
	\label{006}  &\text{The  invertible matrix of every $C$-matrix is also some $C$-matrix;}\\
	\label{007}  &\text{Every $G$-matrix has the row sign-coherent property;}\\
	\label{008}  &\text{The dual mutations of $G$-matrix and $C$-matrix hold.}
\end{align}

These claims were proved by Nakanishi and Zelevinsky \cite{NZ} (see also \cite{N}) by assuming $c$-vectors are sign coherent. Even though the sign coherence  conjecture of $c$-vectors was proved for the acyclic case by Huang and Li  \cite{HL}, and in full generality by Li and Pan \cite{LP},  the proofs  in \cite{NZ} largely depends on the  skew-symmetrizable property. 

Reading studied the property of $G$-fan $\mathcal{F}_{G}$ consisting of all $G$-cones (non-negative spans of the column vectors of $G$-matrices) and their faces in \cite{R}. For the skew-symmetrizable cluster algebras, he also defined the mutation fan, and he proved the following result.
\begin{align}
	\label{009} \text{ $\mathcal{F}_G$ is a fan. Indeed it is a subfan of the mutation fan.}
\end{align}

The bijection of $g$-vectors and cluster variables is a corollary of the positivity conjecture and the above claims as proved in \cite{N}.

\begin{align}
	\label{0099} \text{ Two cluster variables are the same $\Longleftrightarrow$  their corresponding $g$-vectors are the same. }
\end{align}

Fomin and Zelevinsky also made a series of conjectures on the properties of exchange graphs of $\mathcal{A}$ as follows.  
\begin{align}
	\label{010}	&\text{Every seed is uniquely defined by its cluster. }\\
	\label{011}	&\text{Two clusters are adjacent  $\Longleftrightarrow$ they have exactly $n- 1$ common cluster variables.}\\
	\label{012}	&\text{Each seed with principal coefficients is  determined by the corresponding $C$-matrix..}
\end{align}
These conjectures were proved for skew-symmetrizable cluster algebras by Gekhtman, Shapiro, and  Vainshtein \cite{GSV}, Cao and Li \cite{CL},  and Cao, Huang, and Li \cite{CHL}.

In this paper, we show that under an condition (the {\bf Assumption} in Section \ref{sec32}),  the claims (\ref{005}---\ref{012}) hold for totally sign-skew-symmetric cluster algebras. In particular, they hold for acyclic sign-skew-symmetric cluster algebras.

\begin{Theorem}[Proposition \ref{sdual}, \ref{dualmut},  Theorem \ref{Gfan},  \ref{gsubfan},  \ref{aban}, \ref{suiyi} ]
	Let $B$ be  a totally sign-skew-symmetric matrix satisfying the {\bf Assumption}. Consider the corresponding matrix pattern and cluster algebra, then the above claims (\ref{005}---\ref{012}) hold.  In particular, claims (\ref{005}---\ref{012}) hold for acyclic cluster algebras.
\end{Theorem}

In the next section, we recall basic theory of cluster algebras. In Section \ref{ft}, we  describe  behaviors about cluster algebras under folding theory.  In Section \ref{pg}, we prove the properties of $G$-matrices for totally sign-skew-symmetric cluster algebras that satisfy a certain condition (which we conjecture it is always true). In the final section, we  study the properties of exchange graphs.

\section{Preliminaries}	\label{sec2}
Let $I$ be a countable set. We say an integer matrix $B\in Mat_{I\times I}(\mathbb{Z})$  is {\em locally finite} if for each $i\in I$, there are only finitely many nonzero entries $b_{ij}$ and $b_{ji}$, and  say $B$  is  {\em sign-skew-symmetric} if $b_{ij}b_{ji} \leq 0$ and $b_{ij}b_{ji} = 0 \Leftrightarrow b_{ij}=b_{ji}=0$ for all $i, j \in I$. A {\em skew-symmetrizable} matrix $B$ admits a diagonal matrix $D=\text{diag}(d_i, i\in I)$ with positive integer diagonal entries such that $d_ib_{ij} = -d_jb_{ji}$, which  is clearly sign-skew-symmetric. A {\em quiver} is a directed graph without loops or $2$-cycles,  and since a locally finite skew-symmetric matrix naturally corresponds to a locally finite quiver, we do not distinguish them in the following sections.

A semifield $\mathbb{P}$ is a set equipped  with a structure  of  abelian multiplicative group and a structure of abelian  additive semigroup (the addition is usually denoted by $\oplus$) such that the addition $\oplus$  is  distributive with respect to the multiplication in $\mathbb{P}$.  The {\em  tropical semifield} $\mathbb{P}=\mathrm{{trop}}(z_j, j \in J)$ with an index set $J$ is a multiplicative group generated freely by
the elements $z_j, j\in J$,  whose addition $\oplus$ is given by \[\prod_{j\in J}z_j^{a_j}\oplus  \prod_{j\in J}z_j^{b_j}=\prod_{j\in J}z_j^{\mathrm{min}(a_j,b_j)}.\]
Let $\mathbb{QP}$ be the quotient field of the integral domain $\mathbb{ZP}$. 

\begin{Definition}

A  (labeled) seed  is a triplet $\Sigma = (\mathbf{x},\mathbf{y}, B)$, where
\begin{enumerate}
\item[-] $\mathbf{x} =  (x_i, i\in I)$  is  a countable set of indeterminates over $\mathbb{QP}$ such that $\mathcal{F}=\mathbb{QP}(x_i, i\in I)$ is a purely transcendental ﬁeld extension of $\mathbb{QP}$, and $\mathbf{x}$ is called a cluster,  $x_i, i\in I$ are called cluster variables;
\item[-] $\mathbf{y} = (y_i, i\in I)$ is a subset of the semifield $\mathbb{P}$ indexed by $I$;
\item[-] $B\in Mat_{I\times I}(\mathbb{Z})$  is a locally finite sign-skew-symmetric matrix, and it is called the exchange matrix. 
\end{enumerate}

\end{Definition}

 For $k\in I$, define another
triplet $(\mathbf{x'},\mathbf{y'},B') = \mu_k(\mathbf{x},\mathbf{y},B)$ which is called the {\em mutation} of $(\mathbf{x},\mathbf{y},B)$
at $k$ and obtained by the following rules:
\begin{enumerate}
	\item[-] $\mathbf{x'} = (x'_i, i\in I)$ is given by \[x'_i =\begin{cases} x_i, &\text{$i\neq k$},\\  
		 \frac{y_k\prod x_i^{[b_{ik}]_+} + \prod x_i^{[-b_{ik}]_+}}{(y_k \oplus 1)x_k} ,&\text{$i=k$.} \end{cases}\]
	\item[-] $\mathbf{y'} = (y'_i, i\in I)$ is given by \[y'_i =\begin{cases}  y_k^{-1}, &\text{$i=k$,} \\ y_iy_k^{[b_{ki}]_+}(y_k \oplus 1)^{-b_{ki}}, &\text{$i\neq k$.} \end{cases}\]
	\item[-] $B'=(b'_{ij})$ is given by\[b'_{ij}=\begin{cases} -b_{ij},&\text{if $i=k$ or $j=k$,}\\b_{ij} + \mathrm{sign}(b_{ik})[b_{ik}b_{kj}]_+,&\text{otherwise.}\end{cases}\]
\end{enumerate}
where $[a]_+=\mathrm{max}\{a,0\}$. 

Say two labeled seeds $\Sigma$ and $\Sigma'$ define the same unlabeled seed if there is a bijection $\sigma$ of $I$ such that 
$x'_i = x_{\sigma(i)}$, $y'_i = y_{\sigma(i)}$, $b'_{ij} = b_{\sigma(i)\sigma(j)}$ for all $i,j\in I$.

If every matrix obtained from a locally finite sign-skew-symmetric matrix by an arbitrary finite sequence of mutations is also sign-skew-symmetric, we call it a {\em totally sign-skew-symmetric} matrix. Clearly, skew-symmetrizable matrices are totally mutable. 
For a sign-skew-symmetric matrix  $B=(b_{ij})_{i,j\in I}$, we may define a directed graph whose vertices are indexed by $I$, and there is an arrow from $i$ to $j$ if $b_{ij}>0$.  We say $B$ is acyclic if the corresponding directed graph has no oriented cycles.
Huang and Li also proved acyclic sign-skew-symmetric matrices are totally mutable \cite{HL}.

Let $\mathbb{T}$ be an $|I|$-regular tree  and its valencies emitting from each common vertex are labeled by $I$.
A  {\em cluster pattern} is an $|I|$-regular tree $\mathbb{T}$ such that for each vertex $t\in \mathbb{T}$, there is a
 seed $\Sigma_t=(\mathbf{x}_t,\mathbf{y}_t,B_t)$ and for each edge labeled by $k\in I$, two  seeds in the endpoints are obtained from each other by seed mutation at $k$.  Note that a cluster pattern is determined by a totally sign-skew-symmetric matrix. The corresponding {\em exchange graph} is  the graph whose vertices are unlabeled seeds and two unlabeled seeds are connected by an edge if they are related by a single mutation.  
 The {\em cluster algebra} $\mathcal{A} = \mathcal{A}(\mathbf x, \mathbf y, B)$ associated to a totally sign-skew-symmetric matrix $B$ is the $\mathbb{ZP}$-subalgebra of $\mathcal{F}$ generated by all cluster variables on $\mathbb{T}$ corresponding to the matrix $B$. The cluster algebra $\mathcal{A}=\mathcal{A}(\mathbf{x}, \mathbf{y}, B)$ is said to have {\em principal
 	coefficients},  if the semifield is the tropical semifield given by $\mathbb{P}=\mathrm{trop}(y_i, i\in I)$.

Fomin and Zelevinsky proved the following Laurent property for cluster algebras, and they also conjectured cluster variables have non-negative Laurent expressions.

\begin{Theorem}[The Laurent phenomenon, \cite{fomin2002cluster,fomin2007cluster}]\label{laurentphen}
Let $\mathbb{T}$ be a cluster pattern.  Then  for  arbitrary two vertices $t, t_0 \in \mathbb{T}$ and any $j\in I$, we have that 
$x_{j,t} \in \mathbb{ZP}[x_{i,t_0}^{\pm 1}, i\in I]$.  Moreover,
\begin{enumerate}
 \item when $\mathbb{P} = trop(z_i, i\in J)$ is a tropical semifield, we have that $x_{j,t} \in \mathbb{Z}[z_{p}, x_{i,t_0}^{\pm 1}, p\in J, i\in I]$.
 \item when $\mathbb{P} = trop(y_{i,t_0}, i\in I)$ is the principal
case, each cluster variable $x_{j,t}$ is homogeneous with respect to a given $\mathbb{Z}^{|I|}$-grading in $\mathbb{Z}[y_{i,t_0}, x_{i,t_0}^{\pm 1}, i\in I]$, which is given by
\[deg(x_{i,t_0})=\mathbf{e}_i, \,\,\,\,deg(y_i)=-\mathbf{b}_i,\]
where $\mathbf{e}_i, i\in I$ are the standard basis vectors in $\mathbb{Z}^{|I|}$, and $\mathbf{b}_{i,t_0}$ is the $i$-th
column vector of $B_{t_0}$.
 \end{enumerate}
\end{Theorem}

In the following discussion, we shall fix a seed, which we call the {\em initial seed} and is usually  placed at $t_0 \in \mathbb{T}$, to obtain other datum from it.  

Let $\mathcal{A}$ be a cluster algebra with principal coefficients and initial seed $\Sigma = \Sigma_{t_0}$. Then by Theorem \ref{laurentphen}, we have that $x_{j,t} \in  \mathbb{Z}[y_i, x_i^{\pm 1},  i\in I]$.  The polynomial $F_{j,t} :=x_{j,t}|_{x_i=1, i\in I}$ is called the {\em $F$-polynomial} (with the initial seed at $t_0$). Since $y_{j,t}\in \text{trop}(y_i,i\in I)$, we have that $y_{j,t}$ is of the form $y_{j,t} = \prod_{i\in I}y_i^{c_{ij,t}}$ for some integers $c_{ij,t}, i\in I$. The column vector
$c_{j,t} = (c_{ij,t}, i\in I)^T$ is called the {\em $c$-vectors} and the matrix $C^{t_0}_{t} = (c_{j,t})_{j\in I}$ is called the {\em $C$-matrix} at $t$.
 The column vector $g_{j,t}$,  the degree  vector of $x_{j,t}$ with respect to the grading given in Theorem \ref{laurentphen}, is called the {\em $g$-vector},  and
the matrix $G^{t_0}_t=(g_{j,t})_{j\in I}$ is called the {\em $G$-matrix} at $t$. 
Note that the definitions of $F$-polynomials, $C$-matrices and $G$-matrices depend on the choice of the initial seed.
Also in \cite{fomin2007cluster}, Fomin and Zelevinsky obtained the mutation rules for $F$-polynomials, $C$-matrices and $G$-matrices. 

For each edge $t\frac{k}{\quad\quad}t'$ on $\mathbb{T}$, we have that 
\begin{enumerate}
	\item[-] $F_{i,t'}=\begin{cases} F_{i,t},&\text{if $i\neq k$;}\\(F_{k,t})^{-1}(\prod_{j\in I} y_j^{c_{jk,t}}F_{j,t}^{[b_{jk,t}]_+}+\prod_{j\in I}y_j^{-c_{jk, t}}F_{j,t}^{[-b_{jk,t}]_+}) ,&\text{otherwise.}\end{cases}$

	\item[-] $c_{ij, t'}=\begin{cases} -c_{ik, t},&\text{if $j=k$;}\\c_{ij, t} + \mathrm{sign}(c_{ik,t})[c_{ik,t}b_{jk,t}]_+ ,&\text{otherwise.}\end{cases}$

\item[-] $g_{ij, t'}=\begin{cases} g_{ij,t},&\text{if $j\neq k$;}\\-g_{ik,t} + \Sigma_{s\in I}g_{is,t}[-b_{sk,t}]_+- \Sigma_{s\in I}b_{is,t}[-c_{sk,t}]_+,&\text{otherwise.}\end{cases}$
\end{enumerate}

A nonzero  vector is said to be {\em sign-coherent}  if  it has  either all non-negative coordinates or all non-positive coordinates.  Fomin and Zelevinsky conjectured every $c$-vectors is  sign coherent  and  each row vector of  every $G$-matrix is sign-coherent. The sign coherence conjecture were solved for skew-symmetrizable cluster algebras in \cite{gross2018canonical}  and for acyclic sign-skew-symmetric cluster algebras  in \cite{HL,CHL2}.

\section{Folding theory}\label{ft}
Recall that  a quiver is a directed graph (may be infinite) which  has no loops or $2$-cycles. It is well-known that a quiver is naturally associated   to a skew-symmetric matrix.
\begin{Definition}
	
	Let $Q$ be a locally finite quiver with frozen vertices set $F$, and $\Gamma$ be a group acting on the vertex set $Q_0$ of $Q$. The vertices in $Q_0\setminus F$ are called mutable vertices.
Assume that $B=(b_{ij})_{i,j\in Q_0}$ is the matrix corresponding to $Q$,	we say $Q$ is $\Gamma$-admissible, if the following condition are satisfied:
	\begin{enumerate}
		\item[(i)] $i$ is mutable $\iff$ $g(i)$ is mutable,  for all $i\in Q_0$ and $g\in \Gamma$; 
		\item[(ii)]  $b_{ij} = b_{g(i)g(j)}$, for all $i,j\in Q_0$, and $g\in \Gamma$;		
		\item[(iii)] $b_{ig(i)} = 0$, for all mutable $i\in Q_0$, and $g\in \Gamma$;		
		\item[(iv)]  $b_{ij}b_{g(i)j}\geq 0$, for all $i\in Q_0$, $j\in Q_0\setminus F$, and $g\in \Gamma$.				
		\end{enumerate}
	\end{Definition}

\begin{Definition}[Orbit mutation]\label{om}
Assume that $Q$ is $\Gamma$-admissible, and $B=(b_{ij})_{i,j\in Q_0}$ its corresponding matrix. For each mutable $k\in Q_0$, the matrix $B'=(b'_{ij})_{i,j\in Q_0}$ of orbit mutation $Q'=\mu_{[i]}(Q)$ of $Q$ at $[k]$ is given by the following rule:
\[b'_{ij} = \begin{cases} -b_{ij},  &\text{if $i\in[k]$ or $j\in [k]$};\\ b_{ij}+ \sum\limits_{p\in[k]} \frac{|b_{ip}|b_{pj} +b_{ip}|b_{pj}| }{2}, &\text{otherwise.}\end{cases} \]		
\end{Definition}

The following Lemma  follows easily from the definition.
\begin{Lemma}\label{twoacts}
Suppose that there are two actions of $\Gamma$ on the vertices set $Q_0$ making $Q$ $\Gamma$-admissible, and these two actions on the mutable vertices are the same. Then for each mutable vertex $k\in Q_0$, the orbit mutations of $Q$ at $[k]$ with respect to the two actions of $\Gamma$ are the same.	
\end{Lemma}

Assume that $Q$ is $\Gamma$-admissible, denote $[i]$  the $\Gamma$-orbit for each $i\in Q_0$, $\bar{I}$ the  set of $\Gamma$-orbits,  and $\mu_{[i]}  := \prod_{j\in [i]} \mu_j$ the orbit mutation. The matrix $B^{\Gamma} = (b_{[i][j]})_{[i],[j]\in \bar{I}}$ is defined by the following rule:
\[b_{[i][j]} = \sum_{i'\in[i]} b_{i'j}.\]

\begin{Definition}
	The locally finite quiver $Q$ is said to be globally foldable with respect to the group $\Gamma$,  if
	 $Q$ is $\Gamma$-admissible, and for any sequence of orbits $[i_1], \dots, [i_k]$, the quiver $\mu_{i_k}\dots\mu_{i_1}(Q)$ is also $\Gamma$-admissible. In this case, we say $(Q, \Gamma)$ is an unfolding of $B^{\Gamma}$.
	\end{Definition}

The following Lemma follows easily from the definition.

\begin{Lemma}[\cite{FWZ}]\label{fwz5}
	Let $Q$ be a locally finite quiver globally foldable with respect to an action of a group $\Gamma$. Let $\bar{Q}$ be a quiver constructed from $Q$ by introducing new frozen vertices together with some arrows connecting them to the mutable vertices in $Q$. Extend the action of $\Gamma$ from $Q$  to $\bar{Q}$ by making $\Gamma$ ﬁx every newly added vertex. Then the quiver $\bar{Q}$ is globally foldable with respect to $\Gamma$.	
\end{Lemma}

Notice that for a locally finite quiver $Q$, the cluster algebra $\mathcal{A}(\mathbf x, Q)$, $F$-polynomials, $c$-vectors, and $g$-vectors could be defined naturally. If $Q$ is globally foldable with respect to $\Gamma$ and for any mutable  $i\in Q_0$, then define $\mu_{[i]} (x_j)$ by the rules:  $\mu_{[i]} (x_j) = \mu_j(x_j)$ if $j\in [i]$, and  $\mu_{[i]} (x_j) = x_j$ if $j\notin [i]$.  

\begin{Lemma}[\cite{HL}]\label{hlinf}
	Assume that  the locally finite quiver $Q$ is  globally foldable with respect to the group $\Gamma$. Then 
	we have the following results: for any mutable vertices $i_1, \dots, i_s$
	\begin{enumerate}
		\item[(1)] the cluster variables of  $\mathcal{A}(\mu_{[i_s]}\dots\mu_{[i_1]}(\mathbf x, Q))$ is the same as the cluster variables of $\mathcal{A}(\mathbf x,Q)$;
		\item[(2)]  any finite variables in $\mu_{[i_s]}\dots\mu_{[i_1]}(\mathbf x, Q)$ is contained in a cluster of $\mathcal{A}(\mathbf x,Q)$;
		\item[(3)]  any variable in $\mu_{[i_s]}\dots\mu_{[i_1]}(\mathbf x, Q)$ is a cluster variable of $\mathcal{A}(\mathbf x,Q)$;
		\item[(4)] any monomial with variables in  $\mu_{[i_s]}\dots\mu_{[i_1]}(\mathbf x, Q)$ is a cluster monomial of $\mathcal{A}(\mathbf x,Q)$;
		\item[(5)] each $F$-polynomial has the constant term $1$.	
		\end{enumerate}
		\end{Lemma}
\begin{proof}
	The statements $(1)-(4)$ are proved in \cite{HL} for the case $B^{\Gamma}$ is acyclic.  It is easy to see that the proof works for all globally foldable quiver  $Q$.	The claim $(5)$ is 
	proved in  [Lemma 7.12,  \cite{HL}].
	\end{proof}

Let $Q$ be a locally finite quiver without frozen vertices. The corresponding {\em  framed quiver} $\tilde{Q}$  is obtained from $Q$ by adding frozen vertices $Q'_0 := \{i', i\in Q_0\}$ and arrows  $Q'_1 := \{i'\rightarrow i,  i\in Q_0\}$.  If $Q$ is globally foldable with respect to $\Gamma$, by  Lemma \ref{fwz5}, $\tilde{Q}$  is globally foldable with respect to  $\Gamma$ with $\Gamma$ fixing frozen vertices.

\begin{Lemma}\label{sccinf}
If  the locally finite quiver $Q$ without frozen vertices is globally foldable with respect to $\Gamma$,  then 
for any vertices $i_1,\dots,i_s\in Q_0$, there are no two frozen vertices $i', j'$ and a mutable vertex $k$ such that there are arrows of the form $i'\rightarrow k \rightarrow j'$  in  $\mu_{[i_s]}\dots\mu_{[i_1]}(\tilde{Q})$. 
\end{Lemma}
\begin{proof}
	By Lemma \ref{fwz5}, we know that $\tilde{Q}$ is globally foldable with respect to $\Gamma$ with $\Gamma$ fixing frozen vertices, then $\mu_{[i_s]}\dots\mu_{[i_1]}(\tilde{Q})$  is well-defined.
	
	 If the claim is not true,  let us consider the cluster algebra  $\mathcal{A}(\tilde{\mathbf x}, \tilde{Q})$, here $\tilde{\mathbf x} = (\mathbf x, \mathbf y)$ and $\mathbf y = (y_i, i\in Q_0)$ are frozen variables.  Suppose that $(\tilde{\mathbf x}', \tilde{Q}')  =  \mu_{[i_s]}\dots\mu_{[i_1]}(\tilde{\mathbf x}, \tilde{Q})$. By Lemma \ref{hlinf}, the variables in $\tilde{\mathbf x}'$ and the variable 
	\[x'' = \frac{\prod\limits_{i'\rightarrow k} y_i \prod\limits_{i\rightarrow k} x'_{i}  +  \prod\limits_{k\rightarrow j'} y_j \prod\limits_{k\rightarrow j} x'_{j}}{x'_k}\]
	are cluster variables in $\mathcal{A}(\tilde{\mathbf x}, \tilde{Q})$. Under the specialization of $x_i = 1$ for all $i\in Q_0$, we have that
		\[F'' = \frac{\prod\limits_{i'\rightarrow k} y_i \prod\limits_{i\rightarrow k} F'_{i}  +  \prod\limits_{k\rightarrow j'} y_j \prod\limits_{k\rightarrow j} F'_{j}}{F'_k}.\]
By Lemma  \ref{hlinf},	the left side will be $1$ and the right side will be $0$ by letting $y_i=0$ for all $i\in Q_0$. This is a contradiction. 
		\end{proof}

\begin{Lemma}\label{gff}
	Let $Q$ be  a  globally foldable quiver with respect to $\Gamma$, and $\tilde{Q}$ its framed quiver. 
	Extend the action of $\Gamma$ on $Q_0$  to the action of $\Gamma$ on $\tilde{Q}_0 = Q_0$ by $g.(i') := g(i)'$ for each $i\in Q_0$.
	then  $\mu_{[i_s]}\dots\mu_{[i_1]}(\tilde{Q})$ is $\Gamma$-admissible for any sequence of orbit mutations with $i_1,\dots,i_s\in Q_0$.
	\end{Lemma}
\begin{proof}
	We prove the claim by the induction on the length of orbit mutation sequences.
	
	 It is not hard to see that $\tilde{Q}$ is also $\Gamma$-admissible.  Assume that $\mu_{[i_s]}\dots\mu_{[i_1]}(\tilde{Q})$ is $\Gamma$-admissible, we only need to show that $\mu_{[k]}\mu_{[i_s]}\dots\mu_{[i_1]}(\tilde{Q})$ is also 
	$\Gamma$-admissible for every $k\in Q_0$.  
	
		For simplicity, we may abuse notations and let $\begin{pmatrix} B&-C^T\\C&M\end{pmatrix}$ denote the matrix of $\mu_{[i_s]}\dots\mu_{[i_1]}(\tilde{Q})$ and  $\begin{pmatrix} B'&-(C')^T\\C'&M'\end{pmatrix}$ to denote the matrix of $\mu_{[k]}\mu_{[i_s]}\dots\mu_{[i_1]}(\tilde{Q})$, where $C=(c_{i'j})_{i,j\in Q_0}$ and $C'=(c'_{i'j})_{i,j\in Q_0}$. Since $\mu_{[i_s]}\dots\mu_{[i_1]}(\tilde{Q})$ is $\Gamma$-admissible, by Lemma 
     	\ref{twoacts} and Lemma \ref{sccinf},  $C$ and $C'$ have the  column sign-coherent property and $M=M'=0$. Thus we have that $c'_{i'j}c'_{g(i')j} \geq 0$ for  $i,j\in Q_0$.

        Note that $C$ and $C'$ are related by the following relations by Definition :
        \[c'_{i'j}  = \begin{cases} -c_{i'j}, &\text{$j\in [k]$}\\ 
        	                                    c_{i'j} +\sum\limits_{p\in [k]}c_{i'p}b_{pj}, &\text{$j\notin [k], \,\,\exists\, p\in[k]$ such that $c_{i'p}>0, b_{pj}>0;$}\\
        	                                     c_{i'j} +\sum\limits_{p\in [k]}c_{i'p}b_{pj}, &\text{$j\notin [k], \,\,\exists\, p\in[k]$ such that $c_{i'p}<0, b_{pj}<0;$}\\
        	                                     c_{i'j}, &\text{otherwise.}
        	               \end{cases}	 \]
It is easy to check that $c'_{g(i')g(j)} = c'_{i'j}$ for all $i,j \in Q_0$. If $j\in [k]$, then $g(j)\in k$, and hence  $c'_{g(i')g(j)} = -c_{g(i')g(j)} = - c_{i'j} = c'_{i'j}$.  If  $j\notin [k],$ and $\exists\, p\in[k]$, such that $c_{i'p}>0, b_{pj}>0$, then $g(j)\notin [k]$, and $g(p)\in k$ satisfies that $c_{g(i') g(p)} = c_{i'p} >0$,  $b_{g(p)g(j)} = b_{pj}  >0$. Thus we have that 
\[  c'_{i'j} = c_{i'j}  +   \sum\limits_{p\in [k]}c_{i'p}b_{pj}  = c_{g(i')g(j)}  +   \sum\limits_{p\in [k]}c_{g(i')g(p)}b_{g(p)g(j)}   =  c_{g(i')g(j)}  +   \sum\limits_{p\in [k]}c_{g(i')p}b_{pg(j)}  = c'_{g(i')g(j)}.  \]
The other cases can be checked similarly.

This proves that $\mu_{[k]}\mu_{[i_s]}\dots\mu_{[i_1]}(\tilde{Q})$ is also 
$\Gamma$-admissible.
	\end{proof}

\begin{Theorem}\label{anyseed}
	Let $Q$ be a locally finite globally foldable quiver with respect to $\Gamma$. 
	Then  the positivity conjecture holds for the cluster algebra $\mathcal{A}(B^{\Gamma})$, i.e., for every cluster variable $x$ and every cluster $\mathbf x$, we have that $x$ is a Laurent polynomial with variables in $\mathbf x$ and positive coefficients.  
	
	Moreover, each $F$-polynomial has the constant term $1$, and thus the $c$-vectors have the sign-coherent	property.
		\end{Theorem}
	
	\begin{proof}
		It is enough to prove positivity conjecture for cluster algebras with principal coefficients. This follows from Lemma \ref{gff},  Lemma \ref{hlinf} and the fact that positivity conjecture holds for skew-symmetric cluster algebras of infinite rank.		
		\end{proof}
	
	\begin{Theorem}[\cite{HL}]\label{hlunfolding}
		There exists an unfolding  for every sign-skew-symmetric matrix mutation equivalent to an acyclic sign-skew-symmetric matrix.
	\end{Theorem}

Combining Theorem  \ref{anyseed} and Theorem \ref{hlunfolding}, we have the following results.
\begin{Theorem}[cf. \cite{HL}]\label{pc}
	The positivity conjecture holds for the sign-skew-symmetric cluster algebras with the initial seed mutation equivalent to an acyclic seed.	
	\end{Theorem}

\begin{Corollary}\label{printotal}
	If $B=(b_{ij})_{n\times n}$ is a sign-skew-symmetric matrix which is mutation equivalent to an acyclic sign-skew-symmetric matrix, then for any sequence $k_1, \dots, k_s \in [1,n]$, the matrix $\mu_{k_s}\dots\mu_{k_1}\begin{pmatrix} B&-I_n\\I_n&0 \end{pmatrix}$ is sign-skew-symmetric.  
	\end{Corollary} 
\begin{proof}
	By Theorem \ref{hlunfolding}, there is an unfolding $(Q, \Gamma)$ for $B$.  By  Lemma \ref{gff}, there is an unfolding for the matrix $\begin{pmatrix} B&-I_n\\I_n&0 \end{pmatrix}$.

	\end{proof}

\begin{Conjecture}\label{conl}
Corollary \ref{printotal} holds for all totally sign-skew-symmetric matrices.
\end{Conjecture}

\section{The properties of $G$-matrices}\label{pg}
The tropical dualities for $G$-matrices and $C$-matrices were studied by Nakanishi and Zelevinky for the skew-symmetrizable cluster algebras \cite{NZ}. 
Notice that Nakanishi and Zelevinsky's proof depends on the  sign coherence of $c$-vectors. In this section,  we   consider the dual phenomenon for the totally sign-skew-symmetric case. In this section and the next section, we always assume that $|I| = n< \infty$.

As we mentioned in the Introduction, Fomin and Zelevinsky proposed a series of conjectures on $G$-matrices and $C$-matrices in their paper \cite{fomin2007cluster}. Let us recall them.

\begin{Conjecture}\label{fzconj} For the totally sign-skew-symmetric matrix pattern with $B_{t_0} = B$, the following statements are true.
	\begin{enumerate}
		\item[(i)]  For each $t\in \mathbb{T}_n$, the $C$-matrix $C_t$ has the column sign-coherent property.
		\item[(ii)]  For each $t\in \mathbb{T}_n$, the $g$-vectors $g_{1,t}, \dots, g_{n,t}$ form a $\mathbb{Z}$-basis of the lattice $\mathbb{Z}^n$.
		\item[(iii)] For each $t\in \mathbb{T}_n$, the $G$-matrix $G_t$ has the row sign-coherent property.
		\item[(iv)] Let $t_0 \frac{k}{\quad\,}t_1$ be two adjacent vertices in $\mathbb{T}_n$, and let $B^1 = \mu_k(B)$.  Then, for any $t\in \mathbb{T}_n$ and $\mathbf a \in \mathbb{N}^n$, then $g$-vectors $g_{\mathbf a,t}^{B, t_0} := \sum_{l=1}^n a_l g_{l,t}^{B,t_0} = (g_1,\dots, g_n)^T$ and $g_{\mathbf a,t}^{B^1, t_1} := \sum_{l=1}^n a_l g_{l,t}^{B^1,t_1} = (g'_1,\dots, g'_n)^T$  are related as follows:
		\[g'_j = \begin{cases}
			-g_k \quad & j=k;\\
			g_j + [b_{jk}]_+g_k - b_{jk}^0\mathrm{min}(g_k, 0) \quad & j\neq k.
		\end{cases}
		\]
	\end{enumerate}
\end{Conjecture}

In the skew-symmetrizable case, these conjectures are proved by a series of papers \cite{DWZ,gross2018canonical}. In the totally sign-skew-symmetric case,
the Conjectures (i-iv) were proved by Huang and Li \cite{HL} and by Cao, Huang and Li \cite{CHL2} for the acyclic sign-skew-symmetric case  using the folding  theory and categorification, and the Conjecture \ref{fzconj} (i) was proved by Li and Pan in  full generality. In the following sections, we show that under an assumption Conjecture (ii-iv) are true. In particular, for the acyclic sign-skew-symmetric matrix, we obtain these results.  Moreover, we show the duality between $G$-matrices and $C$-matrices, which has been only proved for the skew-symmetrizable case by Nakanishi and Zelevinsky \cite{NZ}. The essential idea of the proofs is developed by Nakanishi and Zelevinsky.

\subsection{Mutations revisited}
The mutations of $G$-matrices and $C$-matrices are given in the Section \ref{sec2}. Now we re-write them in the matrix form.  Let $\mathbb{T}_n$ be the $n$-regular tree and $B$ be a totally sign-skew-symmetric matrix. Then we have an exchange pattern of matrices such that $B_{t_0} = B$.  Now for each vertex $t\in \mathbb{T}_n$, let us replace $B_t$ by a triple $(B_t, C_t, G_t)$, where $C_t$ and $G_t$ are respectively  the corresponding $G$-matrix and $C$-matrices at $t$ with respect to the initial vertex $t_0$. Then we obtain a new pattern and we call it the {\em matrix pattern.} 

Let us introduce some notations we will need in the following. Assume that $A=(a_{ij})_{n\times n}$ is a real matrix, and $k\in [1,n]$.  
\begin{enumerate}
	\item  $[A]_+$   denotes the $n\times n$ matrix obtained from $A$ by replacing $a_{ij}$by $ max\{a_{ij}, 0\}  $ for all $i , j $. 
	\item  $A^{k\cdot}$ denotes the $n\times n$ matrix obtained from $A$ by replacing $a_{ij}$ for $i\neq k$ with $0$.
	\item  $A^{\cdot k}$ denotes the $n\times n$ matrix obtained from $A$ by replacing $a_{ij}$ for $j\neq k$ with $0$.
	\item $J_k$ denotes  the $n\times n$ matrix obtained from the identity matrix $I_n$ by replacing the $k$-th diagonal element $1$  with $-1$.
\end{enumerate}

\begin{Proposition}[\cite{fomin2007cluster, NZ}]\label{gcmf}
	Let $B= B_{t_0}$ be a totally sign-skew-symmetric matrix.  In the matrix pattern, for each 
	$t\frac{k}{\quad\,}t'$ in $\mathbb{T}_n$, we have that 
	\[C_{t'} = C_t(J_k + [\varepsilon B_t]_+^{k\cdot}) + [-\varepsilon C_t]_+^{\cdot k}B_t,\]
	\[G_{t'} = G_t(J_k + [-\varepsilon B_t]_+^{\cdot k}) - B_{t_0}[-\varepsilon C_t]^{\cdot k},\]
	where $\varepsilon \in \{1, -1\}$.
\end{Proposition}

We shall notice that the so-called the {\em first duality} always holds for the totally sign-skew-symmetric matrix pattern.

\begin{Proposition}[\cite{fomin2007cluster}]\label{fdual}
	For each $t\in \mathbb{T}_n$, we have the following equality:
	\[G_tB_t = B_{t_0}C_t.\]
\end{Proposition}
\begin{proof}
	This equality was proved by Fomin and Zelevinsky in \cite{fomin2007cluster} using the fact that $g$-vectors are the $\mathbb{Z}^n$-grading of cluster variables coming from cluster algebras with principal coefficients. Later, Nakanishi proved it using the Proposition \ref{gcmf} and induction on the distance between $t_0$ and $t$ on $\mathbb{T}_n$ in \cite{N}. Their proofs all hold for the totally sign-skew-symmetric case.
\end{proof}

In Li and Pan's latest paper \cite{LP}, they proved the following theorem which affirms the column sign-coherence for all $C$-matrices arising from totally sign-skew-symmetric matrices.

\begin{Theorem}[\cite{LP}]
	Let $B$ be a totally sign-skew-symmetric matrix and $\mathcal{A}(B)$ be the corresponding 
	cluster algebras with principal coefficients. Then every $F$-polynomial has a constant term $1$. Consequently, the $C$-matrices are column sign-coherent.
\end{Theorem}

Since the column sign-coherence holds for $C$-matrices, we can use $\varepsilon_k(C)$ to denote the sign of the $k$-th column of the $C$-matrix $C$. Then we have the following result.

\begin{Proposition}\label{cgmm}
	Let $B= B_{t_0}$ be a totally sign-skew-symmetric matrix.  In the matrix pattern, for each 
	$t\frac{k}{\quad\,}t'$ in $\mathbb{T}_n$, we have that 
	\[C_{t'} = C_t(J_k + [\varepsilon B_t]_+^{k\cdot}) ,\]
	\[G_{t'} = G_t(J_k + [-\varepsilon B_t]_+^{\cdot k}),\]
	where $\varepsilon =  \varepsilon_k(C_t)$.
\end{Proposition}
\begin{proof}
	By  the sign-coherence  of the $C$-matrix, we know that $[-\varepsilon(C_t) C_t]_+^{\cdot k} = 0$. Then this proposition is an easy corollary of the Proposition \ref{gcmf}.
\end{proof}

Now we can see that Conjecture \ref{fzconj} (ii) always holds for totally sign-skew-symmetric matrix pattern.

\begin{Theorem}\label{unim}
	For each $t\in \mathbb{T}_n$,  we have that
	\begin{enumerate}
		\item[(i)] the $g$-vectors $g_{1,t}, \dots, g_{n,t}$ form a $\mathbb{Z}$-basis of the lattice $\mathbb{Z}^n$,
		\item[(ii)]  the $c$-vectors $c_{1,t}, \dots, c_{n,t}$ form a $\mathbb{Z}$-basis of the lattice $\mathbb{Z}^n$.
	\end{enumerate}
	Moreover, we have that $|G_t| = |C_t| \in \{1, -1\}$.
\end{Theorem}

\begin{proof}
	Let us prove $|G_t| = |C_t| \in \{1, -1\}$ by induction on the distance $d(t_0,t)$ between $t_0$ and $t$ on $\mathbb{T}_n$.
	
	For $d(t_0, t) = 0$, then $t = t_0$, we have that $G_t = I_n = C_t$ and thus $|G_t| = |C_t| =1$.  Assume that it holds for $d(t_0, t) =d$  and let  $t\frac{k}{\quad\,}t'$ be an edge on $\mathbb{T}_n$. Then by Proposition \ref{cgmm}, we have that 
	\[C_{t'} = C_t(J_k + [\varepsilon B_t]_+^{k\cdot}) ,\]
	\[G_{t'} = G_t(J_k + [-\varepsilon B_t]_+^{\cdot k}),\]
	where $\varepsilon =  \varepsilon_k(C_t)$.
	Note that \[(J_k + [\varepsilon B_t]_+^{k\cdot})^2= I_n= (J_k + [-\varepsilon B_t]_+^{\cdot k})^2,\] and \[|J_k + [\varepsilon B_t]_+^{k\cdot}| = -1 = |J_k + [-\varepsilon B_t]_+^{\cdot k}|,\]and  $|G_t| = |C_t| \in \{1, -1\}$. We have that  $G_{t'}$ and $C_{t'}$ are invertible in $M_n(\mathbb{Z})$, and 
	$|G_{t'}| = |C_{t'}| \in \{1, -1\}$.
\end{proof}

\subsection{Sign-coherence  and tropical dualities}\label{sec32}
Let $B$ be a totally sign-skew-symmetric matrix and $(B_t, C_t, G_t)_{t\in \mathbb{T}_n}$ be the corresponding matrix pattern  with the initial vertex  $t_0$. It is clear that $B^T$ is also totally sign-skew-symmetric. The corresponding $G$-matrices and $C$-matrices for $B^T$ are denoted  by $\bar{G}$ and $\bar{C}$ respectively.
Note that $\tilde{B} = -B^T$ is also a totally sign-skew-symmetric matrix.  Let $(\tilde{B}_t, \tilde{C}_t, \tilde{G}_t)_{t\in \mathbb{T}_n}$ denote the matrix pattern corresponding to the matrix $\tilde{B} = -B^T$ with the initial vertex  $t_0$. It is easy to see that $\tilde{B}_t = -B_t^T$ for all $t\in \mathbb{T}_n$. 

{\bf Assumption} for $B$:  The signs of the columns of $C$-matrices $C$ and $\tilde{C}$ are always the same for  arbitrary initial vertex for the cluster  algebra $\mathcal{A}(B)$ and for the cluster  algebra $\mathcal{A}(-B)$.

Note that Conjecture \ref{conl} implies the {\bf Assumption}, thus {\bf Assumption} holds for acyclic sign-skew-symmetric matrices.

\begin{Proposition}\label{sdual}
	Suppose the {\bf Assumption} holds. 
	For any  $t_0, t \in \mathbb{T}_n$, we have that 
	\[(G_t^{t_0})^T\tilde{C}_{t}^{t_0} = I_n.\]
	
\end{Proposition}

\begin{proof} Let us prove it by induction on the distance $d(t_0, t)$ of $t_0$ and $t$ on $\mathbb{T}_n$.
	
	It is obvious when $d(t_0, t) = 0$, since in this case $G_t^{t_0} =\tilde{C}_{t}^{t_0} = I_n$
	Suppose that $d(t_0,t) =s$. Let us consider the following sequence:
	\[t_0 \frac{}{\quad\quad\quad\quad\quad\quad}t\frac{k}{\quad\,}t'\]
	By the Proposition \ref{cgmm}, we have 
	\[(G_{t'})^T = (G_t(J_k + [-\varepsilon_k(C_t) B_t]_+^{\cdot k}))^T = (J_k + [-\varepsilon_k(C_t) (B_t)^T]_+^{k\cdot})G_t^T,\]
	\[\tilde{C}_{t'} = \tilde{C}_t (J_k + [-\varepsilon_k(\tilde{C}_t) (B_t)^T]_+^{k\cdot}).\]
	Then \[\begin{split}  (G_{t'})^T\tilde{C}_{t'} &= (J_k + [-\varepsilon_k(C_t) (B_t)^T]_+^{k\cdot})G_t^T\tilde{C}_t (J_k + [-\varepsilon_k(\tilde{C}_t) (B_t)^T]_+^{k\cdot})\\&=  (J_k + [-\varepsilon_k(C_t) (B_t)^T]_+^{k\cdot})(J_k + [-\varepsilon_k(\tilde{C}_t) (B_t)^T]_+^{k\cdot})\\ &= I_n. \end{split}\]
	Here the third equality holds because the {\bf Assumption} ensures that $\varepsilon_k(C_t) = \varepsilon_k(\tilde C_t).$ Then we finish the proof.
\end{proof}

The row sign-coherence for $G$-matrices was proved for the skew-symmetric case by Derksen, Weyman and Zelevinsky in \cite{DWZ}, for the skew-symmetrizable case by Gross, Hacking, Keel and Kontsevich in \cite{gross2018canonical}, and for the  acyclic sign-skew-symmetric case in \cite{CHL2}.
Next, we will prove the row sign-coherence for $G$-matrices under the {\bf Assumption.} The idea of the following proof  comes from the proofs in \cite{NZ} with a little modification.
\begin{Proposition}\label{dualmut}
	Suppose the {\bf Assumption} holds. Then we have the following results:
	
	\begin{enumerate}
		\item[(a)]  For any $t_0, t \in \mathbb{T}_n$, we have
		\begin{equation}
			C_t^{t_0} = (\bar{G}_{t_0}^t)^T,
		\end{equation}
		\begin{equation}
			G_t^{t_0} = (\bar{C}_{t_0}^t)^T.
		\end{equation}
		This implies that the row sign-coherence holds for $G$-matrices, and we use $\varepsilon_k(G)$ to denote the sign of the $k$-th row of the $G$-matrix $G$. 
		\item[(b)] For any $t_0, t_1, t\in \mathbb{T}_n$ such that $t_0 \frac{k}{\quad\,}t_1$ are two adjacent vertices on $\mathbb{T}_n$, we have that 
		\begin{equation}
			C_t^{t_1} = (J_k + [-\varepsilon_k(G^{t_0}_t)B_{t_0}]^{k\cdot}_+)C^{t_0}_t,
		\end{equation}
		\begin{equation}
			G_t^{t_1} = (J_k + [\varepsilon_k(G^{t_0}_t)B_{t_0}]^{\cdot k}_+)G^{t_0}_t.
		\end{equation}
		Here $\varepsilon_k(G^{t_0}_t)$ is well-defined when (a) holds.
		\item[(c)] For any $t_0, t \in \mathbb{T}_n$ and $i, j\in [1,n]$, we have that the $i$-th column of $C_t^{t_0}$ is $\pm e_j$ if and only if the the $i$-th column of $\tilde{C}_t^{t_0}$ is $\pm e_j$.
		\item[(d)] For any $t_0, t \in \mathbb{T}_n$ and $i, j\in [1,n]$, 
		we have that the signs of the $i$-th rows of $G_t^{t_0}$ and $G_t^{t_0}$ are the same,
		and the $i$-th row of $G_t^{t_0}$ is $\pm e_j$ if and only if the the $i$-th row of $\tilde{G}_t^{t_0}$ is $\pm e_j$.
		
	\end{enumerate}
	
\end{Proposition}
\begin{proof}
	We will prove these statements by induction on the distance $d(t_0, t)$ of $t_0$ and $t$ on $\mathbb{T}_n$.
	
	Let us write the following claims:
	\begin{enumerate}
		\item[$(a)_s$.]  The claim $(a)$ holds for any $t_0, t \in \mathbb{T}_n$ such that $d(t_0, t) \leq s$.
		\item[$(b)_s$.]  The claim $(b)$ holds for any $t_0, t_1, t \in \mathbb{T}_n$ such that $d(t_0, t) \leq s$, and $t_1\frac{k}{\quad\,}t_0$ are two adjacent vertices on $\mathbb{T}$ labeled by any $k\in [1,n]$.
		\item[$(c)_s$.]  The claim $(c)$ holds for any $t_0, t \in \mathbb{T}_n$ such that $d(t_0, t) \leq s$.
		\item[$(d)_s$.]  The claim $(d)$ holds for any $t_0, t \in \mathbb{T}_n$ such that $d(t_0, t) \leq s$.
	\end{enumerate}
	
	Let us prove these claims in the following order. 
	
	\begin{equation}\label{abcd}
		(a)_0 \implies (b)_0 \implies (c)_0 \implies (d)_0 \implies (a)_1 \implies (b)_1 \implies \dots
	\end{equation}
	
	Firstly, we show  that  $(a)_0, (b)_0, (c)_0$ and $(d)_0$ hold. Since $s = 0$ implies that $t = t_0$,  thus we have that \[C_t^{t_0} = \bar{G}_{t_0}^t = G_t^{t_0} = \bar{C}_{t_0}^t =\tilde{C}_{t}^{t_0} = \tilde{G}_{t}^{t_0}  = I_n,\] 
	and therefore $(a)_0, (c)_0, (d)_0$ hold naturally. For $(b)_0$, we need to show that 		
	\[C_{t_0}^{t_1} = J_k + [-B_{t_0}]^{k\cdot}_+,\]
	\[G_{t_0}^{t_1} = J_k + [B_{t_0}]^{\cdot k}_+.\]
	By the Proposition \ref{cgmm}, we have that 
	\[C_{t_0}^{t_1} = C_{t_1}^{t_1} (J_k + [B_{t_1}]_+^{k\cdot})= J_k + [-B_{t_0}]^{k\cdot}_+, \]
	\[G_{t_0}^{t_1} =  G_{t_1}^{t_1} (J_k + [-B_{t_1}]_+^{\cdot k})  =J_k + [B_{t_0}]^{\cdot k}_+.\]
	Thus $(b)_0$ holds.
	
	Assume that  the claims in \ref{abcd} up to $(d)_s$, then we show $(a)_{s+1}$.
	Suppose that $d(t_0, t) =s$. Let us consider the following sequence:
	\[t_1\frac{k}{\quad\,}t_0 \frac{}{\quad\quad\quad\quad\quad\quad}t\]
	We need to prove that 
	\[	C_t^{t_1} = (\bar{G}_{t_1}^t)^T,\quad\quad
	G_t^{t_1} = (\bar{C}_{t_1}^t)^T.\]
	By $(a)_s$ and $(b)_s$, we have that 
	\[C_t^{t_0} = (\bar{G}_{t_0}^t)^T,\quad\quad
	G_t^{t_0} = (\bar{C}_{t_0}^t)^T,\]
	\[C_t^{t_1} = (J_k + [-\varepsilon_k(G^{t_0}_t)B_{t_0}]^{k\cdot}_+)C^{t_0}_t,\]
	\[G_t^{t_1} = (J_k + [\varepsilon_k(G^{t_0}_t)B_{t_0}]^{k\cdot}_+)G^{t_0}_t.\]
	Note that $\varepsilon_k(G^{t_0}_t) = \varepsilon(C^{t}_{t_0})$. Then take the transpose, we have 
	\[\begin{split}  (C_t^{t_1})^T &= (C^{t_0}_t)^T(J_k + [-\varepsilon_k(G^{t_0}_t)B_{t_0}^T]^{\cdot k}_+)\\ &= \bar{G}_{t_0}^t (J_k + [-\varepsilon_k(C^{t}_{t_0})B_{t_0}^T]^{\cdot k}_+)\\ & =
		\bar{G}_{t_1}^t.
	\end{split}	
	\]
	The third equality follows from the Proposition \ref{cgmm}. And similarly, we have that
	\[\begin{split}  (G_t^{t_1})^T &= (G^{t_0}_t)^T(J_k + [\varepsilon_k(G^{t_0}_t)B_{t_0}^T]^{k\cdot}_+)\\ &= \bar{C}_{t_0}^t (J_k + [\varepsilon_k(C^{t}_{t_0})B_{t_0}^T]^{k\cdot }_+)\\ & =
		\bar{C}_{t_1}^t.
	\end{split}	
	\]
	Hence $(a)_{s+1}$ holds.
	
	Now assume that  the claims in \ref{abcd} up to $(a)_{s+1}$, then we show $(b)_{s+1}$.
	Assume that $d(t_0, t) =s$. Let us consider the following sequence:
	\[t_1\frac{k}{\quad\,}t_0 \frac{}{\quad\quad\quad\quad\quad\quad}t\frac{l}{\quad\,}t'\]
	We need to prove that 
	\begin{equation}\label{bs1}
		C_{t'}^{t_1} = (J_k + [-\varepsilon_k(G^{t_0}_{t'})B_{t_0}]^{k\cdot}_+)C^{t_0}_{t'},
	\end{equation}
	\begin{equation}\label{bs2}
		G_{t'}^{t_1} = (J_k + [\varepsilon_k(G^{t_0}_{t'})B_{t_0}]^{\cdot k}_+)G^{t_0}_{t'}.
	\end{equation}
	To prove the equality \ref{bs1}, we need to prove that 
	\[		C_{t'}^{t_1} = (J_k + [-\varepsilon_k(G^{t_0}_{t'})B_{t_0}]^{k\cdot}_+)C^{t_0}_{t} (J_l+ [\varepsilon_l(C_t^{t_0}B_t)]_+^{l\cdot}).
	\]
	Note that we have that 
	\[\begin{split}C_{t'}^{t_1} &= C_{t}^{t_1}(J_l+ [\varepsilon_l(C^{t_1}_t)B_t]^{l\cdot}_+) \\& = 
		(J_k+[-\varepsilon_k(G_{t}^{t_0})B_{t_0}]_+^{k\cdot}) C_{t}^{t_0}(J_l+ [\varepsilon_l(C^{t_1}_t)B_t]^{l\cdot}_+) .\end{split}\] 
	The second equality follows from $(b)_s$. So it is enough to prove that 
	\[\begin{split} &(J_k + [-\varepsilon_k(G^{t_0}_{t'})B_{t_0}]^{k\cdot}_+)C^{t_0}_{t} (J_l+ [\varepsilon_l(C_t^{t_0}B_t)]_+^{l\cdot}) \\=&(J_k+[-\varepsilon_k(G_{t}^{t_0})B_{t_0}]_+^{k\cdot}) C_{t}^{t_0}(J_l+ [\varepsilon_l(C^{t_1}_t)B_t]^{l\cdot}_+) .\end{split} \]
	Now let us discuss the signs in detail. We shall notice that
	by the mutation of $G$-matrices, $G^{t_0}_{t'}$ and $G^{t_0}_{t}$ differ only in the $l$-th column, and thus we have the following result:
	\begin{enumerate}
		\item[(1)] $\varepsilon_k(G^{t_0}_{t'}) = \varepsilon_k(G^{t_0}_{t})$ if and only if there is a nonzero entry in the $k$-th row of $G^{t_0}_{t}$ outside the entry at $(k,l)$,
		\item[(2)]  $\varepsilon_k(G^{t_0}_{t'}) = -\varepsilon_k(G^{t_0}_{t})$ if and only if the $k$-th row of 
		$G^{t_0}_{t}$ is $\pm e_l$.
	\end{enumerate}
	On the other hand, by $(b)_s$, we know that $C_{t}^{t_0}$ and $C_{t}^{t_1}$ differ only in the $k$-th row, and thus we have the following result:
	\begin{enumerate}
		\item[(3)] $\varepsilon_l(C^{t_1}_{t}) = \varepsilon_l(C^{t_0}_{t})$ if and only if there is a nonzero entry in the $l$-th column of $C^{t_0}_{t}$ outside the entry at $(l,k)$,
		\item[(4)]  $\varepsilon_l(C^{t_1}_{t}) = -\varepsilon_l(C^{t_0}_{t})$ if and only if the $l$-th column of 
		$C^{t_0}_{t}$ is $\pm e_k$.
	\end{enumerate}
	
	If $\varepsilon_k(G^{t_0}_{t'}) = \varepsilon_k(G^{t_0}_{t})$, we need to show that $\varepsilon_l(C^{t_1}_{t}) = \varepsilon_l(C^{t_0}_{t})$. Indeed, if $\varepsilon_l(C^{t_1}_{t}) = -\varepsilon_l(C^{t_0}_{t})$, then the $l$-th column of 
	$C^{t_0}_{t}$ is $\pm e_k$, and by $(c)_{s}$, we know that the $l$-th column of 
	$\tilde{C}^{t_0}_{t}$ is $\pm e_k$. By Proposition \ref{sdual}, $(G^{t_0}_{t})^T \tilde{C}^{t_0}_{t} = I_n$, this implies that the $k$-th row of 
	$G^{t_0}_{t}$ is $\pm e_l$, and hence $\varepsilon_k(G^{t_0}_{t'}) = \varepsilon_k(G^{t_0}_{t})$, which is a contradiction. Thus if $\varepsilon_k(G^{t_0}_{t'}) = \varepsilon_k(G^{t_0}_{t})$, the equality \ref{bs1} holds.

	If $\varepsilon_k(G^{t_0}_{t'}) = -\varepsilon_k(G^{t_0}_{t})$, then the $k$-th row of 
	$G^{t_0}_{t}$ is $\pm e_l$. Thus by Proposition \ref{sdual} and the Assumption, we know that 
	$l$-th column of 
	$C^{t_0}_{t}$ is $\pm e_k$. Notice that in this case, we have $ \varepsilon_k(G^{t_0}_{t}) = \varepsilon_l(C^{t_0}_{t})$, and hence $\varepsilon_l(C^{t_1}_{t}) = -\varepsilon_l(C^{t_0}_{t})$.  Thus we may let $ \varepsilon_k(G^{t_0}_{t}) = \varepsilon_l(C^{t_0}_{t}) = \varepsilon \in \{1, -1\}$, and $ \varepsilon_k(G^{t_0}_{t'}) = \varepsilon_l(C^{t_1}_{t}) = -\varepsilon$. To prove 
	the equality \ref{bs1}, it suffices to show that 
	\[\begin{split} &(J_k + [\varepsilon B_{t_0}]^{k\cdot}_+)C^{t_0}_{t} (J_l+ [\varepsilon B_t)]_+^{l\cdot}) \\=&(J_k+[-\varepsilon B_{t_0}]_+^{k\cdot}) C_{t}^{t_0}(J_l+ [-\varepsilon B_t]^{l\cdot}_+) ,\end{split} \]
	which is equivalent to prove that 
	\[\begin{split} &C^{t_0}_{t} (J_l+ [\varepsilon B_t)]_+^{l\cdot}) (J_l+ [-\varepsilon B_t]^{l\cdot}_+) \\=&(J_k + [\varepsilon B_{t_0}]^{k\cdot}_+)(J_k+[-\varepsilon B_{t_0}]_+^{k\cdot}) C_{t}^{t_0}.\end{split} \]
	Then compute the above equality, we only need to show that
	\[(B_{t_0}C_{t}^{t_0})^{k\cdot} = C_{t}^{t_0}(B_t)^{l\cdot}.\]
	Note that in this case we have that \[(C_{t}^{t_0})^{\cdot l}  =  \varepsilon E_{kl} =(G_{t}^{t_0})^{k\cdot } .\] Thus we have that 
	\[C_{t}^{t_0}(B_t)^{l\cdot} = (C_{t}^{t_0})^{\cdot l} B_t =        (G_{t}^{t_0})^{k\cdot } B_t =  (G_{t}^{t_0}B_t)^{k\cdot } .  \]
	Here the first equality and the third equality follow from the facts that for any $n\times n$ matrices $P$ and $Q$, and any $k\in [1,n]$, we have that 
	\[P^{\cdot k}Q = P Q^{k\cdot},  \quad\quad P^{k\cdot}Q = (P Q)^{k\cdot}.\]
	By Proposition \ref{fdual}, we have that $(B_{t_0}C_{t}^{t_0})^{k\cdot}=(G_{t}^{t_0}B_t)^{k\cdot } $, and $(B_{t_0}C_{t}^{t_0})^{k\cdot} = C_{t}^{t_0}(B_t)^{l\cdot}$. This proves the equality 
	\ref{bs1} in this case.
	
	Let us prove the equality \ref{bs2}. Similarly, it is enough to show that
	\[\begin{split} &(J_k + [\varepsilon_k(G^{t_0}_{t'})B_{t_0}]^{\cdot k}_+)G^{t_0}_{t} (J_l+ [-\varepsilon_l(C_t^{t_0}B_t)]_+^{\cdot l}) \\=&(J_k+[\varepsilon_k(G_{t}^{t_0})B_{t_0}]_+^{\cdot k}) C_{t}^{t_0}(J_l+ [-\varepsilon_l(C^{t_1}_t)B_t]^{\cdot l}_+) .\end{split} \]
	After similar  discussion as above, the case of $\varepsilon_k(G^{t_0}_{t'}) = \varepsilon_k(G^{t_0}_{t})$ is similar.  When $\varepsilon_k(G^{t_0}_{t'}) = -\varepsilon_k(G^{t_0}_{t})$, as above, we only need to show that
	\[ (B_{t_0})^{\cdot k}G_t^{t_0} = (G_t^{t_0}B_t)^{\cdot l}.\]
	Similarly,  we have that 
	\[(B_{t_0})^{\cdot k}G_t^{t_0} = B_{t_0}(G_t^{t_0})^{k\cdot } =B_{t_0} (C_t^{t_0})^{\cdot l} =(B_{t_0} C_t^{t_0})^{\cdot l}  = (G_t^{t_0}B_t)^{\cdot l}.\]
	Therefore, the equality \ref{bs2} holds in this case. Then $(b)_{s+1}$ holds.
	
	Now assume that  the claims in \ref{abcd} up to $(b)_{s+1}$, then we show $(c)_{s+1}$.
	Suppose that $d(t_0, t) =s+1$. (We shall emphasis that here we assume that $d(t_0, t) =s+1$, while in the previous proof we often assume that $d(t_0, t) =s$.) For any $i, j\in [1,n]$, we need to prove that the $i$-th column of $C_t^{t_0}$ is $\pm e_j$ if and only if the $i$-th column of $\tilde{C}_t^{t_0}$ is $\pm e_j$.  Consider the following sequence 
	\[t_j\frac{j}{\quad\,}t_0 \frac{}{\quad\quad\quad\quad\quad\quad}t\]
	Since $d(t_0, t) = s+1$ and $(b)_{s+1}$ holds, we have that 
	\begin{equation}\label{e1}
		C_t^{t_1} = (J_k + [-\varepsilon_k(G^{t_0}_t)B_{t_0}]^{k\cdot}_+)C^{t_0}_t,
	\end{equation}
	\begin{equation}\label{e2}
		\tilde{C}_t^{t_1} = (J_k + [-\varepsilon_k(\tilde{G}^{t_0}_t)\tilde{B}_{t_0}]^{k\cdot}_+)\tilde{C}^{t_0}_t.
	\end{equation}
	Then we have the following facts:
	\begin{enumerate}
		\item By the equality \ref{e1}, we know that the $i$-th column of $C_t^{t_0}$ is $\pm e_j$ if and only if  $\varepsilon_i(C_t^{t_0}) = - \varepsilon_i(C_t^{t_j})$. 
		\item By the equality \ref{e2}, we know that the $i$-th column of $\tilde{C}_t^{t_0}$ is $\pm e_j$ if and only if  $\varepsilon_i(\tilde{C}_t^{t_0}) = - \varepsilon_i(\tilde{C}_t^{t_j})$. 
		\item By the {\bf Assumption}, we know that $\varepsilon_i(C_t^{t_0}) =\varepsilon_i(\tilde{C}_t^{t_0})$ and $  \varepsilon_i(C_t^{t_j})= \varepsilon_i(\tilde{C}_t^{t_j}).$
	\end{enumerate}
	Therefore the $i$-th column of $C_t^{t_0}$ is $\pm e_j$ if and only if the $i$-th column of $\tilde{C}_t^{t_0}$ is $\pm e_j$.
	
	Now assume that  the claims in \ref{abcd} up to $(c)_{s+1}$, then we show $(d)_{s+1}$.
	
	Suppose that $d(t_0, t) =s+1$.  By $(a)_{s+1}$, we have that \[G_t^{t_0} = (C_{t_0}^{B^T_t, t})^T, \quad \tilde{G}_t^{t_0} = (C_{t_0}^{-B_t, t})^T.\]
	By the {\bf Assumption}, we know that the signs of the columns of the $C$-matrices $C_{t_0}^{B^T_t, t}$ and $ C_{t_0}^{-B_t, t}$ are the same. Thus the signs of the rows of the $G$-matrices $G_t^{t_0} $ and $ \tilde{G}_t^{t_0}$ are the same.  
	
	Note that the $i$-th row of $G_t^{t_0}$ is $\pm e_j$ if and only if the $j$-th column of $\tilde{C}_t^{t_0}$ is $\pm e_i$, and the $i$-th row of $\tilde{G}_t^{t_0}$ is $\pm e_j$ if and only if the $j$-th column of $C_t^{t_0}$ is $\pm e_i$. While by $(c)_{s+1}$,  the $j$-th column of $\tilde{C}_t^{t_0}$ is $\pm e_i$ if and only if the $j$-th column of $C_t^{t_0}$ is $\pm e_i$. Hence we have that the $i$-th row of $G_t^{t_0}$ is $\pm e_j$ if and only if the the $i$-th row of $\tilde{G}_t^{t_0}$ is $\pm e_j$.  This finishes the proof.
\end{proof}

\begin{Proposition} Suppose that the {\bf Assumption} holds.
	Let $t_0 \frac{k}{\quad\,}t_1$ be two adjacent vertices in $\mathbb{T}_n$, and let $B^1 = \mu_k(B)$.  Then, for any $t\in \mathbb{T}_n$ and $\mathbf a \in \mathbb{N}^n$, then $g$-vectors $g_{\mathbf a,t}^{B, t_0} := \sum_{l=1}^n a_l g_{l,t}^{B,t_0} = (g_1,\dots, g_n)^T$ and $g_{\mathbf a,t}^{B^1, t_1} := \sum_{l=1}^n a_l g_{l,t}^{B^1,t_1} = (g'_1,\dots, g'_n)^T$  are related as follows:
	\begin{equation}\label{maini}
		g'_j = \begin{cases}
			-g_k \quad & j=k;\\
			g_j + [b_{jk}]_+g_k - b_{jk}min(g_k, 0) \quad & j\neq k.
		\end{cases}
	\end{equation}
\end{Proposition}
\begin{proof}
	We may write the equation \ref{maini} in the matrix form as follows:
	\[G^{t_1}_t =(J_k + [\varepsilon B]_+^{\cdot k})G_{t}^{t_0} + B[-\varepsilon G_{t}^{t_0}]_+^{k\cdot}, \]
	where $\varepsilon \in \{1,-1\}$. Thanks to Proposition \ref{dualmut}  (a),  we may let $\varepsilon = \varepsilon_k(G_{t}^{t_0})$, and by Proposition \ref{dualmut}  (d), the 
	equation \ref{maini} holds.
\end{proof}

Then we have the following results for acyclic cluster algebras.
\begin{Theorem}\label{acydualmut}
	If $B$ is mutation to an acyclic sign-skew-symmetric matrix. Then we have that 
	\begin{enumerate}
		\item For each $t\in \mathbb{T}_n$ and $k\in [1,n]$, $\varepsilon_k(C_t) = \varepsilon_k(\tilde{C}_t)$.
		\item  For any $ t \in \mathbb{T}_n$ and $i, j\in [1,n]$, we have that the $i$-th column of $C_t^{t_0}$ is $\pm e_j$ if and only if the the $i$-th column of $\tilde{C}_t^{t_0}$ is $\pm e_j$.
		\item  	For any  $t_0, t \in \mathbb{T}_n$, 
		$(G_t^{t_0})^T\tilde{C}_{t}^{t_0} = I_n.$
		\item For any $t_0, t_1, t\in \mathbb{T}_n$ such that $t_0 \frac{k}{\quad\,}t_1$ are two adjacent vertices on $\mathbb{T}_n$, we have that 
		\begin{equation}
			C_t^{t_1} = (J_k + [-\varepsilon_k(G^{t_0}_t)B_{t_0}]^{k\cdot}_+)C^{t_0}_t,
		\end{equation}
		\begin{equation}
			G_t^{t_1} = (J_k + [\varepsilon_k(G^{t_0}_t)B_{t_0}]^{\cdot k}_+)G^{t_0}_t.
		\end{equation}
		\item For each $t\in \mathbb{T}_n$ and $i,j, k\in [1,n]$, $\varepsilon_k(G_t) = \varepsilon_k(\tilde{G}_t)$,
		and the $i$-th row of $G_t^{t_0}$ is $\pm e_j$ if and only if the the $i$-th row of $\tilde{G}_t^{t_0}$ is $\pm e_j$.
	\end{enumerate}
\end{Theorem}

\subsection{$G$-fan and the mutation fan}
In this subsection, we always assume that the {\bf Assumption} holds. Let us recall some basic notions on cones and fans.

\begin{Definition}
	A subset $C \subset \mathbb{R}^n$ is called a convex cone if $av \in C$ and $v+v'\in C$ for any $a\in \mathbb{R}_{>0}$ and $v, v'\in C$. In particular, $C$ is called a convex polyhedral cone if there exists a finite number of elements $v_1, \dots, v_s \in C$ such that 
	\[C= \mathbb{R}_{\geq 0}v_1+ \dots +  \mathbb{R}_{\geq 0}v_s  := \{a_1v_1+ \dots + a_sv_s\,|\, a_i\geq 0,\,\,1\leq i\leq s\}.\]
	The convex polyhedral cone $C$ is called a simplicial  cone, if $v_1, \dots, v_s$ can be chosen to be $\mathbb{R}$-linearly independent.
	
	A convex cone is rational if it is non-negative $\mathbb{R}$-linear span of finitely many integer vectors, or equivalently, it can be defined by finitely many weak linear inequalities with integer coefficients.
\end{Definition}

\begin{Definition}
	A subset $F$ of a convex set $C$ is a face if $F$ is convex and if any line segment $L\subset C$ whose interior intersects $F$ has  $L\subset F$.
\end{Definition}

\begin{Definition}\label{fan}
	A fan $\mathcal{F}$ in $\mathbb{R}^n$ is a collection of closed convex  cones, such
	that 
	\begin{enumerate}
		\item any face of a cone in $\mathcal{F}$ is also a cone in $\mathcal{F}$;
		\item  the intersection of any two cones in $\mathcal{F}$ is a face of each.
	\end{enumerate}
\end{Definition}

In the Definition \ref{fan}, we do not require that the fan is a finite set.

\begin{Definition}
	For a totally sign-skew-symmetric matrix pattern $\{(B_t, C_t, G_t)\}_{t\in \mathbb{T}_n}$, the cone 
	\[\sigma(G_t) := \mathbb{R}_{\geq 0} g_{1,t} + \dots +  \mathbb{R}_{\geq 0} g_{n,t} \]
	is  called a $G$-cone.  The interior of $\sigma(G_t)$ is denoted by $\bar{\sigma}(G_t)$. 
\end{Definition}

Since the $G$-matrix $G_t$ is invertible, we have that 
\[\bar{\sigma}(G_t) := \mathbb{R}_{> 0} g_{1,t} + \dots +  \mathbb{R}_{> 0} g_{n,t} .\]
For a totally sign-skew-symmetric matrix pattern $\{(B_t, C_t, G_t)\}_{t\in \mathbb{T}_n}$ with the initial vertex $t_0$, let us introduce a notation as follows:
\[\mathcal{F}_{G^{t_0}} := \{\sigma(G_t^{t_0}) \,\,\text{and its faces}, t\in \mathbb{T}_n\}.\]

For $k\in [1,n]$, there are two half spaces of $\mathbb{R}^n$ defined as follows:
\[e_k^+ := \{w=(w_1,\dots,w_n)\in \mathbb{R}^n\,|\, w_k\geq 0\}, \quad e_k^- := \{w=(w_1,\dots,w_n)\in \mathbb{R}^n\,|\, w_k\leq 0\}. \]
Let $t_0 \frac{k}{\quad\,}t_1$ be two adjacent vertices in $\mathbb{T}_n$.  Define the piece-linear map $\eta_{t_0}^{t_1}$ as follows:
\[\begin{split} \eta_{t_0}^{t_1}: & \mathbb{R}^n \rightarrow \quad\quad \mathbb{R}^n \\
	& w \mapsto \begin{cases}     (J_k + [B_{t_0}]_+^{\cdot k})w   \quad & w\in  e_k^+ ;\\
		(J_k + [-B_{t_0}]_+^{\cdot k})w   \quad & w\in  e_k^-.  \end{cases}
\end{split}\]
Note that $ \eta_{t_0}^{t_1}$ is well-defined, since if $w_k=0$, we have that $(J_k + [B_{t_0}]_+^{\cdot k})w =w= (J_k + [-B_{t_0}]_+^{\cdot k})w$.
By the Proposition \ref{dualmut}, we have that  for any $t\in \mathbb{T}_n$, 
\[G_t^{t_1} = (J_k + [\varepsilon_k(G^{t_0}_t)B_{t_0}]^{\cdot k}_+)G^{t_0}_t,\]
which means that
\[g_{i,t}^{t_1} = (J_k + [\varepsilon_k(G^{t_0}_t)B_{t_0}]^{\cdot k}_+)g^{t_0}_{i,t}, \forall i\in [1,n].\]
Thus  we have that 
\[ \eta_{t_0}^{t_1} (g^{t_0}_{i,t}) = g_{i,t}^{t_1} \]
for any $i\in[1,n]$ and $t\in \mathbb{T}_n$.

\begin{Proposition}[cf, \cite{N,R}]\label{eta} The following statements hold.
	\begin{enumerate}
		\item $\eta_{t_0}^{t_1} \eta_{t_1}^{t_0} = id = \eta_{t_1}^{t_0}\eta_{t_0}^{t_1}$.
		\item  $\eta_{t_0}^{t_1}(\sigma(G_t^{t_0})) = \sigma(G_t^{t_1})$.
		\item the map \[  \begin{split}    \eta_{t_0}^{t_1}: & \mathcal{F}_{G^{t_0}}  \rightarrow  \mathcal{F}_{G^{t_0}} \\&    g^{t_0}_{i,t} \mapsto g_{i,t}^{t_1}   \end{split}  \] is a bijection between $\mathcal{F}_{G^{t_0}} $ and  $ \mathcal{F}_{G^{t_0}}$ preserving the intersection and the inclusion of cones.		\end{enumerate}
\end{Proposition}

\begin{proof}
	By Proposition \ref{dualmut}, we know that the row sign-coherence holds for $G$-matrices,
	thus we have that either $\sigma(G_t^{t_0}) \subset e_k^+$ or  $\sigma(G_t^{t_0}) \subset e_k^-$. Thus $\eta_{t_0}^{t_1}$ is linear on $\sigma(G_t^{t_0})$ and note that $\eta_{t_0}^{t_1} (g^{t_0}_{i,t}) = g_{i,t}^{t_1}$, we have that 
	\[ \eta_{t_0}^{t_1}(\sigma(G_t^{t_0})) = \sigma(G_t^{t_1}).\]
	
\end{proof}

\begin{Theorem}[cf.\cite{N,R}]\label{Gfan}
	Suppose that the {\bf Assumption} holds. Then the following set 
	\[\mathcal{F}_{G^{t_0}} := \{\sigma(G_t) \,\,\text{and its faces}, t\in \mathbb{T}_n\}\]
	is a fan.
\end{Theorem}

\begin{proof}
The proof is similar to the proof  of  [Theorem 3.17, \cite{N}].
\end{proof}

\begin{Corollary}
	
	For any $t\in \mathbb{T}_n$ and $i\in [1,n]$, if the $g$-vector $g_{i,t}$ is non-negative, then $g_{i,t} = e_j$ for some $j$, and consequently, the cluster variable $x_{i,t}$ satisfies that $x_{i,t} = x_{j,t_0}$. 
	
\end{Corollary}

\begin{Definition}
	Let $B$ be a  totally sign-skew-symmetric matrix, and let $\{(B_t, C_t, G_t)\}_{t\in \mathbb{T}_n}$ be the corresponding matrix pattern. Two vectors $w, w' \in \mathbb{R}^n$ are said to be sign-equivalent (with respect  to $B$), if  \[sign(\eta_{t_0}^{t}(w))  =   sign(\eta_{t_0}^{t}(w'))\] for any $t\in \mathbb{T}_n$. Here for $w\in \mathbb{R}^n$, $sign(w) := (sign(w_1),\dots, sign(w_n))$.
	
	Sign-equivalence defines an equivalence relation. We call the equivalence classes $B$-classes. 
\end{Definition}

Reading proved that the closure of any $B$-class are convex cones.
\begin{Lemma}[\cite{R}]
	
	The closures of $B$-classes are convex cones.
	
	We call them  $B$-cones.
\end{Lemma}

\begin{Definition}
	Let $\mathcal{MF}(B)$ be the collection consisting of all $B$-cones, together with all faces of $B$-cones. This collection $\mathcal{MF}(B)$ is called the mutation fan for $B$. 
\end{Definition}

\begin{Theorem}[\cite{R}]
	The mutation fan $\mathcal{MF}(B)$  is a complete fan.
\end{Theorem}

Next, following Reading, we show that the $G$-fan is a subfan of the mutation fan.  Note that  we always assume that the {\bf Assumption} holds for $B$ appearing in this section. 
\begin{Definition}
	Two full dimensional cones are adjacent if they have a common face of codimension $1$ and they have disjoint interiors. 
	
	We say two full dimensional cones $C$ and $C'$  in a fan $\mathcal{F}$ are transitively adjacent  if there is a sequence $C=C_0, C_1, \dots, C_k =C'$ of full-dimensional cones in $\mathcal{F}$ such that $C_{i=1}$ and $C_i$ are adjacent for  all  $i\in [1, k]$. The full dimensional cones in $\mathcal{MF}(B)$ that are transitively adjacent to nonnegative cone  $\mathbb{R}^n_{\geq 0}$ in $\mathcal{MF}(B)$ are the maximal  cones of a subfan  $\bar{\mathcal{MF}}(B)$   of   $\mathcal{MF}(B)$ .
	
\end{Definition}

\begin{Theorem}[cf. \cite{R}]\label{gsubfan}
	For a totally sign-skew-symmetric matrix $B$, the following fans are the same:
	\[\bar{\mathcal{MF}}(B) =  \mathcal{F}_{G^{t_0}}.\]
\end{Theorem}

\begin{proof}
	The proof is similar to the proof in \cite{R}.
\end{proof}

\section{Exchange graphs}
Recall that Fomin and Zelevinsky proposed the following conjecture:

\begin{Conjecture}\label{cldseed}
	Every seed is uniquely defined by its cluster; thus, the vertices of the exchange graph can be identified with the clusters, up to a permutation of cluster variables.
	
\end{Conjecture}

\begin{Conjecture}\label{common}
	Two clusters are adjacent in the exchange graph if and only if they have exactly $n- 1$ common cluster variables.
\end{Conjecture}

\begin{Conjecture}\label{cdeter}
	Let $\bar{\mathcal{A}}$ be the cluster algebra with principal coefficients. Then each cluster  $(\bar{\mathbf x}_t, \bar{\mathbf y}_t, B_t)$ is  uniquely determined by the corresponding $C$-matrix $C_t$.
\end{Conjecture}

These two conjectures were proved by Gekhtman, Shapiro, and  Vainshtein for the skew-symmetrizable cluster algebras of geometric type \cite{GSV}, and they had shown that if Conjecture  \ref{cldseed} holds, then the Conjecture \ref{common} holds. Cao and Li proved Conjecture  \ref{cldseed} for generalized cluster algebras with any coefficients \cite{CL}.
Cao, Huang and Li proved Conjecture \ref{cdeter} for the skew-symmetrizable cluster algebras \cite{CHL2}. In this section, we prove these three conjectures for cluster algebras whose exchange matrices satisfying the {\bf Assumption} and the positivity conjecture. These conditions  hold for  acyclic sign-skew-symmetric cluster algebras. We also give a proof of Conjecture \ref{cldseed} and Conjecture \ref{common} for cluster algebras of odd rank and with indecomposable exchange matrix without the {\bf Assumption} condition.

The following theorem is powerful, since it establishes the one-to-one  correspondence between $g$-vectors and cluster variables with any coefficients. 
\begin{Theorem}\label{sss}
	Let $\mathcal{A}$ be a totally sign-skew-symmetric cluster algebra with coefficients coming from an semifield $\mathbb{P}$ satisfying the {\bf Assumption}. Then  for any $t_0 \in \mathbb{T}_n$, we have the following bijective  map
	\[ \begin{split}
		\{\text{$g$-vectors}\,\, g_{i,t}^{t_0},  \,\, i\in [1,n],\,\, t\in\mathbb{T}_n          \} & \longrightarrow   \{\text{cluster variables $x_{i,t} $  in $\mathcal{A} $ }, \,\, i\in [1,n],\,\, t\in\mathbb{T}_n   \}     \\
		g_{i,t}^{t_0}    &\mapsto x_{i,t}.
	\end{split}   \]
\end{Theorem}
\begin{proof}
	
See the proof of [Theorem 8.2, \cite{N}] for the details.
\end{proof}

Notice that Theorem \ref{sss} was also proved in full generality for cluster algebras with principal coefficients by  Li and Pan in their recent paper \cite{LP}.

\begin{Definition}
	Let $\mathcal{A}$ be a totally sign-skew-symmetric cluster algebra with coefficients coming from an arbitrary semifield $\mathbb{P}$. The cluster complex  $\mathcal{C}(\mathcal{A})$ is the simplicial complex whose vertices are cluster variables and whose simplicies are non-empty subsets of clusters.

\end{Definition}

We have the following results:

\begin{Corollary}\label{ssss}
	Let $\mathcal{A}$ be a totally sign-skew-symmetric cluster algebra with coefficients coming from an arbitrary semifield $\mathbb{P}$. Then for any $t_0\in \mathbb{T}_n$, we have an isomorphism of simplicial complices:
	\[\begin{split}  \mathcal{C}(\mathcal{A})  &\longrightarrow \mathcal{F}_{G^{t_0}}\\
		x_{i,t}   & \leftrightarrow g_{i,t}^{t_0}. 
	\end{split}\]
	Here $\mathcal{F}_{G^{t_0}}$ is the corresponding $G$-fan.
\end{Corollary}

Before we prove the main results, let us make some preparations. 
Recall the following formula for the universal coefficients for cluster algebras with totally sign-skew-symmetric matrix.

\begin{Theorem}[\cite{fomin2007cluster}]\label{sqy}
	Let $\mathcal{A}$ be a totally sign-skew-symmetric cluster algebra with coefficients coming from the universal semifield $\mathbb{P} = \mathbb{Q}_{sf}(Y_1,\dots,Y_n)$.  Then we have that  for any $j\in [1,n]$ and $t\in \mathbb{T}_n$,
	\[Y_{j,t}   = (\prod_{i=1}^n Y_i^{c_{ij,t}} )  \prod_{i=1}^n F_{i,t}(\mathbf Y)^{b_{ij,t}},\]
	where $F_{i,t}$ are $F$-polynomials.
	
\end{Theorem}

For a permutation $\sigma \in S_n$ and a seed $(\mathbf x_t, \mathbf y_t, B_t)$,  the action of $\sigma$ on the seed $(\mathbf x_t, \mathbf y_t, B_t)$,  is denoted by $\sigma(\mathbf x_t, \mathbf y_t, B_t) $ and given by 
\[\sigma(\mathbf x_t) := (x_{\sigma(1),t}, \dots, x_{\sigma(n),t}),\,\sigma(\mathbf y_t) := (y_{\sigma(1),t}, \dots, y_{\sigma(n),t}) , \sigma(B_t) := (b_{\sigma(i)\sigma(j),t}). \]

\begin{Theorem}\label{aban}
	Let $\mathcal{A}$ be a totally sign-skew-symmetric cluster algebra with coefficients coming from an semifield $\mathbb{P}$. Suppose that   the {\bf Assumption} holds for the 
	exchange matrix of $\mathcal{A}$. 
	Then we have that:
	\begin{enumerate}
		\item Every seed is uniquely defined by its cluster, i.e., if there is a permutation $\sigma\in S_n$ such that $\sigma(\mathbf x_{t'}) = \mathbf x_{t}$ for some $t, t'\in \mathbb{T}_n$, then \[    (\mathbf x_t, \mathbf y_t, B_t) =  \sigma(\mathbf x_{t'}, \mathbf y_{t'}, B_{t'}).                 \]
		\item  Two clusters are adjacent in the exchange graph if and only if they have exactly $n- 1$ common cluster variables.
	\end{enumerate}
\end{Theorem}

\begin{proof}
	(1)  Assume that  $\sigma(\mathbf x_{t'}) = \mathbf x_{t}$ for some $t, t'\in \mathbb{T}_n$.
	We may let $t'= t_0$ for simplicity.   
	Then by Theorem  \ref{sss},  we have that  there is a
	bijective  map
	\[ \begin{split}
		\{\text{$g$-vectors}\,\, g_{i,t}^{t_0},  \,\, i\in [1,n],\,\, t\in\mathbb{T}_n          \} & \longrightarrow   \{\text{cluster variables $x_{i,t} $  in $\mathcal{A} $ }, \,\, i\in [1,n],\,\, t\in\mathbb{T}_n   \}     \\
		g_{i,t}^{t_0}    &\mapsto x_{i,t}.
	\end{split}   \]
	Let us consider the cluster algebra $\bar{\mathcal{A}}(\bar{\mathbf x},\bar{\mathbf y}, B )$ with principal coefficients simultaneously.
	
	Since $\sigma(\mathbf x_{t_0}) = \mathbf x_{t}$, thus the $G$-matrix $G^{t_0}_t$ is a  permutation matrix $P$ (every column and every row of $P$ only have one element  $1$.)
	This implies that the $F$-polynomials $F_{i,t}^{t_0}$ are $1$, and the matrix $\tilde{C}_{t}^{t_0}$ is also $P$, as $P^TP = I_n$ and $(G^{t_0}_t)^T\tilde{C}_{t}^{t_0 }= I_n$.  By the Proposition \ref{dualmut}, we know that ${C}_{t}^{t_0 } = \tilde{C}_{t}^{t_0 } =P$. 
	
	By the Proposition \ref{fdual}, we have that \[ B_t = (G^{t_0}_t)^{-1}B_{t_0} {C}_{t}^{t_0 } = P^TB_{t_0} P = \sigma(B_{t_0}).\]
	
	By the Theorem \ref{sqy}, we have that $Y_{j,t} = Y_{\sigma(j)}$ for each $j\in[1,n]$, and thus $y_{j,t}  =  y_{\sigma(j)}$ for each $j\in[1,n]$, i.e., $\mathbf y_t = \sigma(\mathbf y_{t_0}).$
	
	(2) It is a Corollary  of (1). 
\end{proof}

Next, we abandon the {\bf Assumption}, and prove that Theorem \ref{aban} holds for all totally sign-skew-symmetric cluster algebras of odd rank.

\begin{Theorem}
	Let $\mathcal{A}$ be a totally sign-skew-symmetric cluster algebra of odd rank with coefficients coming from an semifield $\mathbb{P}$ and with indecomposable exchange matrix. Then we have that:
	\begin{enumerate}
		\item Every seed is uniquely defined by its cluster, i.e., if there is a permutation $\sigma\in S_n$ such that $\sigma(\mathbf x_{t'}) = \mathbf x_{t}$ for some $t, t'\in \mathbb{T}_n$, then \[    (\mathbf x_t, \mathbf y_t, B_t) =  \sigma(\mathbf x_{t'}, \mathbf y_{t'}, B_{t'}).                 \]
		\item  Two clusters are adjacent in the exchange graph if and only if they have exactly $n- 1$ common cluster variables.
	\end{enumerate}
\end{Theorem}

\begin{proof}
	Let $t'= t_0$ for simplicity. Without loss of generality, we may assume that $\sigma = id$.
	
	We also consider the cluster algebra $\bar{\mathcal{A}}(\bar{\mathbf x},\bar{\mathbf y}, B )$ with principal coefficients simultaneously.  By Theorem \ref{sss}, we have that
	\[ \sigma(\mathbf x_{t_0}) = \mathbf x_{t} \implies  g_{i,t}^{t_0} = g_{i,t_0}^{t_0} ,\forall i \implies \sigma(\mathbf{\bar{x}}_{t_0}) = \mathbf{\bar{x}}_{t}\] 
	We also know that the $G$-matrix $G^{t_0}_t$ is the identity matrix $I_n$.   For any $k\in [1,n]$, we have that 
	\[\begin{split}  \mu_k(x_{k,t}) &= \frac{\prod y_i^{[c_{ik,t}]_+}  \prod x_i^{[b_{ik,t}]_+}   +   \prod y_i^{[-c_{ik,t}]_+}  \prod x_i^{[-b_{ik,t}]_+}  }{x_{k}} \\
		& =     \frac{\prod y_i^{[c_{ik,t}]_+}  \prod x_i^{[b_{ik,t}]_+}   +   \prod y_i^{[-c_{ik,t}]_+}  \prod x_i^{[-b_{ik,t_0}]_+}  }{     \prod y_i^{[c_{ik,t_0}]_+}  \prod x_i^{[b_{ik,t_0}]_+}   +   \prod y_i^{[-c_{ik,t_0}]_+}  \prod x_i^{[-b_{ik,t_0}]_+}     }\mu_k(x_{k,t_0})   .  \end{split}                \]
	Since $\mu_k(x_{k,t}) $  is a Laurent polynomial of $x_1, \dots, \mu_{k}(x_{1,t_0}), \dots, x_n$ with coefficients in $\mathbb{N}[y_1, y_2, \dots, y_n]$, thus we have that for each $k$, there is a number $\varepsilon_k \in \{1, -1\}$ such that 
	\[\mathbf b_{k,t} =\varepsilon_k \mathbf b_{k, t_0},\quad \mathbf c_{k,t} =\varepsilon_k \mathbf c_{k, t_0}.\]
	Since the exchange matrix is indecomposable and sign-skew-symmetric, we have that 
	$\varepsilon_1 = \dots =\varepsilon_n \in \{1,-1\}$. Thus we have that 
	either  \[B_t= -B_{t_0}, \quad C_t= -C_{t_0},\]
	or \[B_t= B_{t_0},\quad C_t= C_{t_0}.\]
	Since the cluster algebra is of odd rank, i.e., $n$ is odd. Then if $C_t= -C_{t_0}$, we have that $|C_t| = (-1)^n = -1$. Note that we have proved that $G_t= I_n$, thus $|G_t| = 1$. While  by Proposition \ref{unim}, we have that $|G_t| = |C_t|$.  Thus $|C_t| = -1$ is a contradiction. Thus we have that $B_t= B_{t_0}$, $C_t= C_{t_0} = I_n$.  
	
	By Theorem \ref{sqy} and the fact that $ \mathbf{\bar{x}}_{t_0} = \mathbf{\bar{x}}_{t}$, we know that $\mathbf y_t =\mathbf y_{t_0}$. This finishes the proof.
\end{proof}

\begin{Theorem}\label{suiyi}
	Let $\bar{\mathcal{A}}$ be the cluster algebra with principal coefficients such that the {\bf Assumption} holds for the exchange matrix. Then each cluster  $(\bar{\mathbf x}_t, \bar{\mathbf y}_t, B_t)$ is  uniquely determined by the corresponding $C$-matrix $C_t$.
\end{Theorem}
\begin{proof}
	If there are two vertices $t, t'$ on $\mathbb{T}_n$ such that  $C_t  = C_{t'}$. Let us show that  $(\bar{\mathbf x}_t, \bar{\mathbf y}_t, B_t) =(\bar{\mathbf x}_{t'}, \bar{\mathbf y}_{t'}, B_{t'})$.   Note that by Proposition \ref{sdual},  $C_t  = C_{t'}  \implies    \tilde{G}_t  = \tilde{G}_{t'} $. While by Theorem \ref{sss},   $ \tilde{G}_t  = \tilde{G}_{t'}  \implies \bar{\tilde{{\mathbf x}}}_{t'} = \bar{\tilde{{\mathbf x}}}_{t'}.$ By Theorem \ref{aban}, we have that $ \tilde{C}_t  = \tilde{C}_{t'}$. Thus $G_t  = G_{t'}$ and $\bar{{\mathbf x}}_{t'} = \bar{{\mathbf x}}_{t'}.$ By Theorem \ref{sss},  we have that $(\bar{\mathbf x}_t, \bar{\mathbf y}_t, B_t) =(\bar{\mathbf x}_{t'}, \bar{\mathbf y}_{t'}, B_{t'})$. 
\end{proof}

\begin{Corollary}
	If the totally sign-skew-symmetric matrix satisfies the {\bf Assumption}, then the $G$-matrix and  the $C$-matrix  are uniquely determined by each other, and they determine the seeds uniquely.
\end{Corollary}

As a corollary, we obtain the following results.

\begin{Theorem}
	\begin{enumerate}
		\item[(1)]	Let $\mathcal{A}$ be an acyclic sign-skew-symmetric cluster algebra with coefficients coming from an semifield $\mathbb{P}$.  Then
	\begin{enumerate}
	\item Every seed is uniquely defined by its cluster, i.e., if there is a permutation $\sigma\in S_n$ such that $\sigma(\mathbf x_{t'}) = \mathbf x_{t}$ for some $t, t'\in \mathbb{T}_n$, then \[    (\mathbf x_t, \mathbf y_t, B_t) =  \sigma(\mathbf x_{t'}, \mathbf y_{t'}, B_{t'}).                 \]
	\item  Two clusters are adjacent in the exchange graph if and only if they have exactly $n- 1$ common cluster variables.
\end{enumerate}
		
		\item[(2)] Let $\bar{\mathcal{A}}$ be an acyclic sign-skew-symmetric cluster algebras with principal coefficients. Then  each cluster  $(\bar{\mathbf x}_t, \bar{\mathbf y}_t, B_t)$ is  uniquely determined by the corresponding $C$-matrix $C_t$.
	\end{enumerate}
\end{Theorem}

{\bf Acknowledgements:}\; This work was supported by the National Natural Science Foundation of China (Grants No.12071422).

\bibliographystyle{alpha}
\bibliography{acyclusteralg.bib}

\end{document}